\documentclass[conference]{IEEEconf}

\pdfoutput=1                                                          % needed if you want to
\usepackage{graphicx}
\usepackage{epstopdf}

\usepackage{amsthm}
\usepackage{amsmath}
\usepackage{amssymb}
\usepackage{color}
\usepackage{caption}
\usepackage{subcaption}
\usepackage{enumerate}
\usepackage{url}
\usepackage{algorithm}
\usepackage{algorithmic}
\usepackage{cite}
\usepackage{tikz}
\usepackage{pgfplots}
\usetikzlibrary{plotmarks}

\graphicspath{{figures/}}

\newtheorem{theorem}{Theorem}
\newtheorem{corollary}{Corollary}
\newtheorem{lemma}{Lemma}

\newtheorem{proposition}{Proposition}

\newtheorem{definition}{Definition}
\newtheorem{example}{Example}

\usepackage{comment}
\usepackage{ifthen}
\newboolean{showcomments}
\setboolean{showcomments}{true}
\newcommand{\ex}[1]{  \ifthenelse{\boolean{showcomments}}{ {\mathbb{E}}#1} {}  }
\newcommand{\var}[1]{  \ifthenelse{\boolean{showcomments}}{ {\mathrm{var}}#1} {}  }
\newcommand{\id}[1]{  \ifthenelse{\boolean{showcomments}}{ {\bf 1}_{#1} } {}  }
\newcommand{\eqdef}{:=}
\newcommand{\one}{\mathbf{1}}
\newcommand{\prob}{\mathbb{P}}

\usepackage{ifthen}

\newcommand{\lingwen}[1]{  \ifthenelse{\boolean{showcomments}}
{ \textcolor{red}{(Lingwen says:  #1)}} {}  }
\newcommand{\adam}[1]{\ifthenelse{\boolean{showcomments}}
{ \textcolor{red}{(Adam says:  #1)}}{}}
\newcommand{\steven}[1]{\ifthenelse{\boolean{showcomments}}
{ \textcolor{red}{(Steven says: #1)}}{}}
\newcommand{\niangjun}[1]{\ifthenelse{\boolean{showcomments}}
{ \textcolor{red}{(Niangjun says:  #1)}}{}}
\newcommand{\addcite}[0]{\ifthenelse{\boolean{showcomments}}
{ \textcolor{red}{(addcite)}}{}}
\newcommand{\addcites}[0]{\ifthenelse{\boolean{showcomments}}
{ \textcolor{red}{(addcite(s))}}{}}
\newcommand{\addref}[0]{\ifthenelse{\boolean{showcomments}}
{ \textcolor{red}{(addref)}}{}}
\newcommand{\todo}[1]{\ifthenelse{\boolean{showcomments}}
{ \textcolor{red}{(To do: #1)}} {} }
\newcommand{\delete}[1]{\ifthenelse{\boolean{showcomments}}
{\comment{#1}}{}}
\newcommand{\add}[1]{\ifthenelse{\boolean{showcomments}}
{{\comment{#1}}}{}}
  \renewcommand{\baselinestretch}{.93}
 \newcommand{\nedit}{\color{black}}

\title{Distributional Analysis for \\ Model Predictive Deferrable Load Control}
\author{Niangjun Chen, Lingwen Gan, Steven H. Low, Adam Wierman}

\begin{document}
\maketitle

\begin{abstract}
Deferrable load control is essential for handling the uncertainties associated with the increasing penetration of renewable generation. Model predictive control has emerged as an effective approach for deferrable load control, and has received considerable attention.  Though the average-case performance of model predictive deferrable load control has been analyzed in prior works,  the distribution of the performance has been elusive.  In this paper, we prove strong concentration results on the load variation obtained by model predictive deferrable load control.  These results highlight that the typical performance of model predictive deferrable load control is tightly concentrated around the average-case performance.%, among other things, the impact of short-range and long-range dependencies in the prediction errors.
\end{abstract}

\section{Introduction}

The electricity grid is at the brink of change. On the generation side, the penetration of wind and solar in the energy portfolio is on the rise due to environmental concerns. And, on the demand side, many smart appliances and devices with adjustable power consumption levels are entering the market. The combination of these two changes make generation  less controllable and load less predictable, which makes the traditional ``generation follows load'' model of control much more difficult.

Fortunately, while smart devices make demand forecasting more challenging, they also provide an opportunity to mitigate the intermittency of wind and solar generation from the load side by allowing for demand response. There are two major categories of demand response, direct load control (DLC) and price-based demand response. See \cite{dr_mechanisms} for a discussion of the contrasts between these approaches.

In this paper we focus on \textit{direct load control} with the goal of using demand response to \textit{reduce variations of the aggregate load}.  This objective has been studied frequently in the literature, e.g., \cite{system_loss, gan2013}, because reducing the variations of the aggregate load corresponds to minimizing the generation cost of the utilities.  In particular, large generators with the smallest marginal costs, e.g., nuclear generators and hydro generators, have limited ramp rates, i.e., their power output cannot be adjusted too quickly. So, if load varies frequently, then it must be balanced by more expensive generators (i.e., ``peakers'') that have fast ramp rate.  Thus, if the load variation is reduced, then the utility can use the least expensive sources of power generation to satisfy the electricity demand. %{\nedit Furthermore, for a given total energy consumption over a time horizon, minimizing any convex objective w.r.t power consumption, e.g., power loss as studied in \cite{ev_model}, is equivalent to minimizing load variation.}

\subsection{Model predictive deferrable load control}

%There are a variety of algorithmic challenges inherent in designing a direct load control algorithm for minimizing the load variation. Two of the most fundamental challenges are scalability and good performance in the presence of uncertainties. %\comment{the following.  First, the algorithm needs to scale well with the number of controllable devices, since the expectation is that the number of loads participating in such demand response programs will be large.  Second, the algorithm needs to perform well in the face of uncertain predictions about future loads, renewable generation, etc.}

There is a growing body of work on direct load control algorithms, which includes both simulation-based evaluations \cite{Acha10,Mets10,ilic2002potential} and theoretical performance guarantees \cite{Ma10,Gan12}. The most commonly proposed framework for algorithm design from this literature is, perhaps, model predictive control.

Model predictive control (MPC) is a classical control algorithm, e.g., see \cite{qin2003survey} for a survey. {\nedit MPC can be applied to settings where unknown disturbances to the system are present through the robust control paradigm or the certainty equivalence principle, e.g., see \cite{camacho2013model, kwon2006receding, bemporad1999robust}.} In the context of direct load control, many variations have been proposed.  {\nedit{Scalability and performance in the presence of uncertainty are essential to MPC algorithms for direct load control.} }At this point, there exist model predictive deferrable load control algorithms that can be fully distributed with guaranteed convergence to optimal deferrable load schedules, e.g., \cite{gan2013}.

However, to this point, the evaluation of model predictive deferrable load control has focused primarily on average-case analysis, e.g., \cite{conejo2010real,roos1998industrial}, or worst-case analysis, e.g., \cite{chen2012iems,li2011line}.  While such analysis provides important insights, there is still much to learn about the performance of model predictive deferrable load control.

For example, it is likely that an algorithm has good average performance but bad worst case performance, and vice versa. 
%though the average-case performance of model predictive deferrable load control is good, this alone says little about the ``typical'' performance of the algorithm.  If the variability is high or the distribution is skewed, then bad events may still be quite likely.  It is this sort of thinking that has motivated the study of worst-case analysis in this context; however worst-case analysis is often quite loose when compared with the typical performance.  Instead, 
What is really needed is a distributional analysis that tells us about the ``typical'' performance, which can say, e.g., that the load variation will be less than the desired level 95 percent of the time.  But, to this point, no results on the distribution of the load variation under model predictive deferrable load control exist.

\subsection{Contributions of this paper}

\textit{The main contribution of this paper is to provide a distributional analysis of the load variation under model predictive deferrable load control.}  More specifically, we prove sharp concentration results for the load variation arising from model predictive distributed load control.  %These results can, e.g., be used to bound the probability that the load variation exceeds a certain value and to understand what ``typical'' load variations will look like.

Our results are derived in the context of a standard formulation of the so-called ``optimal deferrable load control'' (OLDC) problem, where %prediction error of future loads are modeled using a Weiner filter \cite{gan2013}.  The model is introduced in Section \ref{s.model}.  We also 
we adopt the model predictive deferrable load control mechanism in \cite{gan2013} since it can be fully distributed, and  average-case analysis suggests that it performs well in environments with uncertain predictions. %Details of the model predictive deferrable load control algorithm are introduced in Section \ref{s.algorithm}.

However, in Proposition \ref{prop: worst-case}, we provide a new worst case analysis which states that this model predictive deferrable load control can be as bad as having no control at all if predictions are adversarial. 

Given this context, the main result of the paper is Theorem \ref{thm: tail-bound}, which proves a Bernstein-type concentration for the load variation under model predictive deferrable load control.  This result highlights that the load variation is concentrated around its mean, and therefore the typical performance is tightly concentrated around the average performance.  Additionally, the result provides useful performance bounds on, e.g., the 95th percentile. 

%Note that it is useful to contrast the distributional analysis in Theorem \ref{thm: tail-bound} with worst-case bounds.  Since worst-case bounds have not previously been derived for the exact setting we consider, we provide a new worst-case analysis in Proposition \ref{prop: worst-case}.  In contrast with the insight derived from the distributional analysis, the worst-case analysis simply states that model predictive deferrable control can be as bad as having no control at all if predictions are adversarial.

Finally, in addition to the usefulness of Theorem \ref{thm: tail-bound} in the context of deferrable load control, the proof technique we develop may also be useful for understanding the distributional performance of model predictive control in other settings.  %In particular, to obtain tight bounds we use a Log-Sobolev inequality to bound the moment generating function of the load variation.  However, a challenge is that the most commonly used approach -- the martingale bounded difference approach, e.g., \cite{mcdiarmid1989} -- applies only if the function does not change much when one of the random variable is substituted by its identical copy.  Unfortunately, this is not the case in our context, so we develop a new approach that, instead of bounding the target by step-wise changes, bounds its entropy using its gradient. This allows us to exploit more structure of the problem and obtain tighter bounds.  Note that this approach should be applicable to the analysis of model predictive control in many other contexts too.

\section{Model}
\label{s.model}

In this paper we consider a standard model for deferrable load control introduced by \cite{ng1998direct} and then studied in, e.g., \cite{ilic2002potential, gan2011optimal, Ma10}.  It is a discrete-time model where the time-slot length matches the timescale at which the power grid system operator makes control decisions.

The goal is to flatten the aggregate load over the control horizon $t \in \{1,...,T\}$. In practice, the control horizon could be a day and a time slot could be on the order of minutes. To formalize the objective of flattening the aggregate load, previous work has tended to focus on minimizing {\nedit the variation of the load:}
\begin{equation} V := \frac{1}{T}\sum_{t=1}^T\left(d(t) - \frac{1}{T}\sum_{\tau=1}^Td(\tau)\right)^2,
\label{eqn: def-load-variance}
\end{equation}
where $d = (d(1), d(2), \ldots, d(T))$ is the aggregate load profile at each time slot.

Importantly, the aggregate load consists of two types. The first type, which is called \textit{baseload}, includes loads like lighting and heating, and is stochastic and non-controllable. Note that renewable generation like wind and solar can be considered as a negative stochastic and non-controllable load.  Denote the baseload by $b=(b(1), b(2), \ldots, b(T))$, and note that $b$ can be interpreted as the difference between non-deferrable load and renewable generation during each time period.

The second type of load, which is called \textit{deferrable load}, consists of devices whose power consumption can be controlled by the utility, e.g., pool pumps, dryers, and electric vehicles taking part in direct load control programs \cite{ev_model, smart_home}.  It is the control of these devices that can be used to minimize \eqref{eqn: def-load-variance}, provided that energy constraints and charging rate constraints are satisfied.  To model deferrable load we consider $N$ devices indexed $1,2,\ldots,N$, and let $p_n(t)$ denote the power consumption of device $n$ at time $t$ for $n=1,2,\ldots, N$ and $t=1,2,\ldots,T$.  Further, each device has associated constraints on the power consumption as follows
\begin{subequations}
\begin{align}
\label{eqn: continuous}
&\underline{p}_n(t) \le p_n(t) \le \bar{p}_n(t), \\
\label{eqn: no-dissipation}
&\sum_{t=1}^T p_n(t) = P_n.
\end{align}
\end{subequations}
Note that, using the above, arrival and deadline constraints can be specified by setting $\underline{p}_n(t) = \bar{p}_n(t) = 0$ for $t$ before arrival and after deadline. {\nedit Here we assume that the deferrable loads are continuously adjustable in constraint \eqref{eqn: continuous} and the power loss due to heat dissipation can be ignored in constraint \eqref{eqn: no-dissipation}. Similar assumptions are made for EV loads in \cite{gan2013, ev_model}. Although real appliances may deviate from these assumptions, we keep these simplifying assumptions as a first step towards analyzing MPC load control algorithm in the presence of uncertainty.   
}

Given the previous notation, we can now formally specify the optimal deferrable load control (ODLC) problem that is the focus of this paper. Define $[k]:=\{1,2,\ldots,k\}$ for $k\in\mathbb{Z}^+$.
	\begin{align} \label{odlc}
	{\text{\bf ODLC: } } \min~~ & \frac{1}{T} \sum_{t=1}^T\left(d(t)-\frac{1}{T}\sum_{\tau=1}^T d(\tau)\right)^2 \\
	\text{over}~~ & p_n(t),d(t), \quad \forall n,t\nonumber\\
	\text{s.t.}~~ & d(t)=b(t)+\sum_{n=1}^N p_n(t), \quad t\in[T];\nonumber \\
	& \underline{p}_n(t) \leq p_n(t) \leq \overline{p}_n(t),\quad n\in[N],t\in[T]; \nonumber \\
	& \sum_{t=1}^T p_n(t) = P_n, \quad n\in[N]. \nonumber
	\end{align}

An important observation is that ODLC is a convex optimization problem, but cannot be solved in real time since the optimal decision at time $t$ depends on future information about the baseload and the arrivals of deferrable load. This information is not known exactly, but commonly there do exist predictions of future baseload and deferrable load arrivals.  So, in practice such predictions are used for real time control.

Thus, the final component of the model is to specify a model for the predictions.  Crucially, prediction errors should grow as prediction is made further into the future.  Further, it is likely that errors are correlated, e.g., an underestimate for time slot $t+1$ likely leads to an underestimate for time slot $t+2$.  To capture these issues, \cite{gan2013} has suggested a model based on Weiner filters, and we adopt the same assumptions here.

Specifically, baseload $b$ is modeled as a random deviation $\delta b$ around its expectation $\bar{b}$ as illustrated in Fig. \ref{fig: base load}. The process $\delta b$ is modeled as a sequence of independent random variables $e(1),\ldots,e(T)$, each with mean 0 and variance $\sigma^2$, passing through a causal filter with impulse response $f$ ($f(\tau)=0$ for $\tau<0$), i.e.,
	\begin{equation*}\label{delta b}
	\delta b(\tau) = \sum_{m=1}^T e(m)f(\tau-m), \qquad \tau=1,\ldots,T.
	\end{equation*}
Using the current information, one can update the prediction at time $t$ by%\footnote{This prediction algorithm is a Wiener filter \cite{wiener}.}
\begin{equation}\label{prediction}
	b_t(\tau) = \bar{b}(\tau) + \sum_{m=1}^te(m)f(\tau-m), \qquad \tau=1,\ldots,T.
	\end{equation}

Further, deferrable loads are modeled as random arrivals over time. Let $N(t)$ be the number of loads that arrive before (or at) time $t$ for $t = 1,...,T$.  Define
 \[a(t) \eqdef \sum_{n=N(t-1)+1}^{N(t)} P_n,\qquad t=1,\ldots,T\]
as the energy request of deferrable loads that arrive at time $t$. {\nedit{We model the total energy request at each time due to arrival of deferrable loads}} $\{a(t)\}_{t=1}^T$  to be a sequence of independent random variables with mean $\lambda$ and variance $s^2$. Further, let $A(t)\eqdef\sum_{\tau=t+1}^T a(\tau)$ denote the total energy requested after time $t$ for $t=1,2,\ldots,T$. 
\begin{figure}[h]
	\centering
      \includegraphics[width=0.4\textwidth]{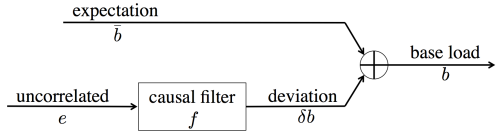}
      	\caption{Diagram of the structure of the baseload model. }
      	\label{fig: base load}
\end{figure}

In summary, when attempting to solve ODLC, an algorithm has, at time $t$, the following information: (i)  {\nedit the energy request and power consumption bounds of }the present deferrable loads, i.e., $\underline{p}_n$, $\overline{p}_n$, and $P_n$ for $n\leq N(t)$, {\nedit  with $\bar{p}_n(t) = \underline{p}_n(t)=0$ for any $t$ beyond the consumption deadline}; (ii) the expectation $\ex(A(t))$ of future energy requests; and (iii) the prediction $b_t$ of the non-deferrable load $b$.

\section{Model predictive deferrable load control}
\label{s.algorithm}

A natural approach for solving the optimal deferrable load control (ODLC) problem described in the previous section is model predictive control, which has been applied in many settings, e.g., see \cite{qin2003survey} for a survey. %In the context of deferrable load control, model predictive control is perhaps the most commonly suggested framework for algorithm design.

%In general, model predictive control uses the available predictions about future to solve an optimization problem over the remainder of the control horizon, and implements only the first step in the solution obtained.  

In the context of the ODLC problem, at each time $t$, such an approach uses the updated prediction of baseload $b_t$ and the updated prediction of future energy request $\ex[A(t)]$ to solve an optimization problem over the remainder of the control horizon, and obtains deferrable load profiles $(p_n(t), p_n(t+1), \ldots, p_n(T))$ for the remainder $\{t,t+1,\ldots,T\}$ of the control horizon. Only $p_n(t)$ will be implemented at time $t$, and $p_n(t+1), \ldots, p_n(T)$ will be recomputed in the future with more updated predictions.

Interestingly, previous work has found that the optimization problem that is solved should not simply be a truncated version of the ODLC problem {\nedit  as done in receding horizon control (RHC)}.  Instead, \cite{gan2013} suggests introducing a pseudo load $q$ to account for the future arrival of deferrable load, {\nedit and plan for the remainder of the entire horizon, giving rise to the shrinking horizon variant of model predictive control.}  The introduction of this term allows for strong analytic guarantees on performance \cite{gan2013}.  Hence, this is the version of model predictive control we consider in this paper.

Specifically, we consider the model predictive deferrable load control algorithm described in Algorithm \ref{algorithm: uncertainty},
    \begin{algorithm}[h]
	\caption{Model Predictive Deferrable Load Control}
	\label{algorithm: uncertainty}
    Initialize $P_n(1)\leftarrow P_n$ for $n=1,2,\ldots,N$;\\
    At time step $t=1,\ldots,T$,
    \begin{algorithmic}[1]
    \STATE Update predictions $b_t$ and $A(t)$;
    \STATE Solve {\bf ODLC-t}$\left(b_t, A(t), \left[P_n(t), \overline{p}_n, \underline{p}_n\right]_{n\in[N(t)]}\right)$ to obtain time-$t$ power consumptions $p_n(t)$ for deferrable loads $n\leq N(t)$ that have already arrived;
    \STATE Update $P_n(t+1)\leftarrow P_n(t)-p_n(t)$ for $n\leq N(t)$;
    \end{algorithmic}
	\end{algorithm}
where at each time $t$ the following optimization problem is solved
    %\begin{subequations}
    \begin{align*}
	&\text{\bf ODLC-t  $\left(b_t, A(t), \left[P_n(t), \overline{p}_n, \underline{p}_n\right]_{n\in[N(t)]}\right)$ } \\
	\min~~ & \sum_{\tau=t}^T \left(\sum_{n=1}^{N(t)}p_n(\tau)+q(\tau)+b_t(\tau)\right)^2 \\
	\text{over}~~ & p_n(\tau),q(\tau),\quad n\leq N(t), \tau\geq t \nonumber\\
	\text{s.t.}~~
	& \underline{p}_n(\tau) \leq p_n(\tau) \leq \overline{p}_n(\tau),\quad n\leq N(t),~\tau\geq t; \nonumber\\
	& \sum_{\tau=t}^T p_n(\tau) = P_n(t), \quad n\leq N(t);\nonumber\\
	& \underline{q}(\tau) \leq q(\tau) \leq \overline{q}(\tau), \quad \tau\geq t;\nonumber\\
	& \sum_{\tau=t}^T q(\tau) = \mathbb{E}(A(t)),\nonumber
	\end{align*}
    %\end{subequations}
In this formulation, $P_n(t) = P_n - \sum_{\tau=1}^{t-1}p_n(\tau)$ is the energy to be consumed at or after time $t$, for all $n$ and all $t$. {\nedit Here $q$ can be viewed as ``pseudo-load'' with the constraint that it sums to the expected future energy request $\mathbb{E}(A(t))$. The constraints $\underline{q}$, $\overline{q}$ are predicted values of maximum and minimum energy request from historical data with $\underline{q}(t)=\overline{q}(t)=0$. However, if no prediction is available, we can simply set $\underline{q}(\tau) = 0$ and $\bar{q}(\tau) = \mathbb{E}(A(t))$ without affecting the theoretical guarantees of the algorithm.}

%{\color{red} This algorithm can be seen as a variant of Model Predictive Control: at every time $\tau$, energy consumption is planned for $t = \tau, \ldots, T$, and the consumption at $t=\tau$ is implemented. However, instead of a constant rolling horizon, here the planning horizon is shrinking as we get closer to $T$.}

Importantly, if predictions are exact then Algorithm \ref{algorithm: uncertainty} solves ODLC exactly.  Further, prior papers have shown that Algorithm \ref{algorithm: uncertainty} can be run in a completely distributed manner and still ensure (fast) convergence to optimal solutions \cite{gan2013}.

%\begin{proposition}[\cite{gan2013}]
%\label{thm: offline}
%When there is no uncertainty, i.e., $N(t)=N$ and $b_t=b$ for $t=1,2,\ldots,T$, then the sequence of deferrable load schedules $(p^{(1)}, p^{(2)}, \ldots, p^{(k)}, \ldots)$ obtained by Algorithm \ref{algorithm: uncertainty} converge to optimal schedules of ODLC.
%\end{proposition}

For our purposes, the most relevant part of previous studies of Algorithm \ref{algorithm: uncertainty} is that there exists simple characterizations of the solutions to ODLC-$t$, which prove quite useful when analyzing the performance of the algorithm.

Specifically, in cases where there are a large number of deferrable loads, the solutions to ODLC-$t$ satisfy a property that is referred to as $t$-valley-filling.

\begin{definition}
For any time $t=1,\ldots,T$, a feasible schedule $(p,q)$ is called \textbf{$t$-valley-filling}, if there exists $C(t)\in\mathbb{R}$ such that
\begin{equation}\label{bounded model}
\sum_{n=1}^{N(t)} p_n(\tau) + q(\tau) + b_t(\tau) =C(t), \qquad \tau=t,\ldots,T.
\end{equation}
\end{definition}

\begin{proposition}[\cite{gan2013}]
\label{thm: characterization}
At time $t=1,\ldots,T$, a $t$-valley-filling deferrable load schedule, if it exists, solves ODLC-$t$.
\end{proposition}

This characterization provides a strong basis for the performance analysis of Algorithm \ref{algorithm: uncertainty}.  To see this, note that if there exists a $t$-valley-filling solution then, besides being optimal, it ensures that the aggregate load satisfies
\begin{equation}
	\label{model}
	d(t) =\frac{1}{T-t+1}\left( \sum_{n=1}^{N(t)} P_n(t) + \ex(A(t)) + \sum_{\tau=t}^Tb_t(\tau) \right)
	\end{equation}
for $t=1,2,\ldots,T$.  This property {\nedit tend to be satisfied when the penetration of deferrable load is high, and it gives us a nice structure to analyze the load variance obtained by Algorithm \ref{algorithm: uncertainty}. Subsequently, we assume that a $t$-valley-filling exists for each $t$ throughout the paper}.

\section{Performance analysis}

The main focus of this paper is the performance analysis of model predictive deferrable load control (Algorithm \ref{algorithm: uncertainty}).  As discussed, the algorithm has been introduced in \cite{gan2013} followed by the \textit{average-case} performance analysis.  The goal of this paper is to perform a \textit{distributional analysis}, rather than simply average-case analysis.  However, to provide context we first introduce the previous average-case analysis and contrast it with a (novel) worst-case analysis.

%Recall that, throughout, we measure performance via the load variance $V$ obtained by the algorithm, i.e., \eqref{eqn: def-load-variance} and we focus on the case where there are enough deferrable loads in order to ensure that a $t$-valley-filling solution exists.

%%%%%%%%%%%%%%%%%%%%%
\subsection{Average-case analysis (previous work)}

An average-case analysis of Algorithm \ref{algorithm: uncertainty} was performed in \cite{gan2013}.  The following is the main result from that paper.

\begin{proposition}[\cite{gan2013}]\label{prop: expectedV}
If a $t$-valley-filling solution exists for $t=1,2,\ldots,T$, then the expected load variation obtained by Algorithm \ref{algorithm: uncertainty} is
\begin{equation}\label{expected load variance}
\ex(V) = \frac{s^2}{T}\sum_{t=2}^T\frac{1}{t} + \frac{\sigma^2}{T^2}\sum_{t=0}^{T-1}F^2(t)\frac{T-t-1}{t+1}.
\end{equation}
where $F(t):=\sum_{m=0}^t f(m)$ for $t=0,\ldots,T$.
\end{proposition}
 Proposition \ref{prop: expectedV} explicitly %states how the generation prediction error ($\sigma$) and deferrable load prediction error ($s$) affect the expected load variation $\ex(V)$. Further, it 
highlights that $\ex(V)\to 0$ as the predictions get precise, i.e., $\sigma\rightarrow0$ and $s\rightarrow0$. More importantly, it follows from Proposition \ref{prop: expectedV} that $\ex(V)$ tends to 0 as time horizon $T$ increases, provided that the error correlation $f(t)$ decays sufficiently fast with $t$.

\begin{proposition}[\cite{gan2013}]\label{prop: large T}
If \eqref{model} holds, and the error correlation $f \sim O(t^{-\frac{1}{2}-\alpha})$ for some $\alpha >0$, then $\ex(V) \rightarrow 0$ as $T \rightarrow \infty$.
\end{proposition}

%We highlight the above proposition because the condition on $f$ reappears in our distributional analysis of $V$.  

This condition is practically relevant since the error correlation $f(t)$ usually decays fast with $t$ and the time horizon $T$ is usually long, which implies that Algorithm \ref{algorithm: uncertainty} should typically have good average case performance.

\subsection{Worst-case analysis}

The results surveyed above highlight that Algorithm \ref{algorithm: uncertainty} performs well on average; however,  it is often important to guarantee more than average case performance.  For that reason, many results in the literature focus on worst case, e.g., \cite{lee1997worst, lin2012online, bemporad1999robust}.  
While no existing results apply directly to the setting of this paper, we can show that the worst-case performance of Algorithm \ref{algorithm: uncertainty} is quite bad.

To see this, let us consider a setting where the prediction error for generation, $e$, and deferrable load, $a$, have bounded deviations from their means (0 and $\lambda$ respectively).

\begin{definition}
We say that \textbf{prediction errors are bounded} if there exist $\epsilon_1$ and $\epsilon_2$ such that, at any time $t=1, \ldots, T$,
\begin{equation}
	\label{infinity norm}
	|a(t) - \lambda| \le \epsilon_1, ~|e(t)|\leq \epsilon_2.
	\end{equation}
\end{definition}

In this situation, it is straightforward to see that the worst case performance of Algorithm \ref{algorithm: uncertainty} can potentially be quite bad. For $a,b\in\mathbb{R}$, define $a\vee b :=\max\{a,b\}.$

\begin{proposition}
If a $t$-valley-filling solution exists for $t=1,2,\ldots,T$, and prediction errors are bounded by $\epsilon_1$ and $\epsilon_2$ as in \eqref{infinity norm}, then the worst-case load variation $\sup_{a,e}V$ achieved by Algorithm \ref{algorithm: uncertainty} is
%\[
%\sup_a V \ge \frac{d_1^2}{4T} \left[\sum_{k=1}^T \left(\ln\left(\frac{T}{k}\right)\right)^2 - T\right] \\
%\approx \frac{d_1^2 }{4}
%\]
	\begin{align*}
	\sup_{a,e}~ V &~=~ \epsilon_1^2 \left(1- \frac{1}{T}\sum_{k=1}^T\frac{1}{k}\right) \\
	& \qquad + \frac{\epsilon_2^2}{T^2} \sum_{\tau=0}^{T-1}\sum_{s=0}^{T-1} \left(\frac{T}{\tau\vee s+1}-1\right) |F(\tau)F(s)|.
	\end{align*}
\label{prop: worst-case}
\end{proposition}
The worst-case performance is achieved when all prediction errors has the maximum magnitude with the appropriate signs---the case where $a(t)=\lambda+\epsilon_1$ and $e(t)=\epsilon_2\cdot\mathrm{sgn}(F(T-t))$ for all $t$. {\nedit  The proof of this proposition can be found in the technical report \cite{chen2014}.}

\begin{corollary}
If a $t$-valley-filling solution exists for $t=1,2,\ldots,T$, and prediction errors are bounded by $\epsilon_1$ and $\epsilon_2$ as in \eqref{infinity norm}, then the worst-case load variation $\sup_{a,e}V$ achieved by Algorithm \ref{algorithm: uncertainty} is lower bounded as
	\[ \sup_{a,e}~ V
	~\geq~ \epsilon_1^2 \left(1- \frac{1}{T}\sum_{k=1}^T\frac{1}{k}\right)
	~\approx~ \epsilon_1^2 \left(1- \frac{\ln T}{T}\right). \]
\label{cor: worst-case}
\end{corollary}
Interestingly, the form of Corollary \ref{cor: worst-case} implies that, in the worst-case, Algorithm \ref{algorithm: uncertainty} can be as bad as having no control at all: the time averaged load variation behaves like the worst one step load variation. Meanwhile, recall from Proposition \ref{prop: large T} that the average performance $\ex(V)\rightarrow0$ as $T\rightarrow\infty$. Hence, while the the load variation $V$ has a small mean $\ex(V)$, it can be quite large in the worst case.

%	\begin{figure}[t]
%	\centering
%      	\includegraphics[width=0.4\textwidth]{different_tails}
%      	\caption{Diagram of the notation and structure of the model for generation.}
%      	\label{fig: long tail}
%	\end{figure}
	
\section{Distributional analysis}
\label{s.analysis}

The contrast between the worst-case analysis (Proposition \ref{prop: worst-case}) and average-case analysis (Proposition \ref{prop: expectedV}) motivates the main goal of this paper --- to understand how often the ``bad cases,'' where $V$ takes large values, happen.  That is, we want to understand what the typical variations of $V$ obtained by Algorithm \ref{algorithm: uncertainty} look like.

\subsection{Concentration bounds}
We start with analyzing the tail probability of $V$. Concretely, our focus is on
\[ V_\eta := \min \{ c\in\mathbb{R} \mid V\leq c \text{ with probability }\eta \}, \]
which denotes the minimum value $c$ such that $V\leq c$ with probability $\eta$ for $\eta\in[0,1]$. Our main result provides upper bounds on $V_\eta$, for large values of $\eta$, for arbitrary of prediction error distributions.

More specifically, we prove that \textit{with high probability}, the load variation of Algorithm \ref{algorithm: uncertainty} does not deviate much from its average-case performance, i.e., we prove a concentration result for model predictive deferrable load control.

\begin{theorem}
Suppose a $t$-valley filling solution exists for $t=1,2,\ldots,T$, and prediction errors bounded by $\epsilon_1$ and $\epsilon_2$ as in \eqref{infinity norm}. %\footnote{Note that it is possible to slightly relax the bounded error assumption and allow the one-step prediction error to be unbounded but with sufficiently light-tail. To do this one uses the fact that the convex Log-Sobolev inequality can be satisfied by sub-Gaussian distributions in addition bounded distributions.\lingwen{I cannot follow the footnote}}
Then the distribution of the load variation $V$ obtained by Algorithm \ref{algorithm: uncertainty} satisfies a Bernstein type concentration, i.e.,
\begin{equation}
\label{tail-bound}
\mathbb{P}(V - \ex{V} > t) \le  \exp\left(\frac{-t^2}{16\epsilon^2 \lambda_1 (2\ex V + t)}\right)
\end{equation}
where $\epsilon = \max(\epsilon_1, \epsilon_2)$ and
\[\lambda_1 = \max\left(\frac{\ln T}{T} , ~ \frac{1}{T^2}\sum_{t=0}^{T-1}F^2(t)\frac{T-t+1}{t+1}\right). \]
\label{thm: tail-bound}
\end{theorem}

{\nedit The theorem is proved in the technical report \cite{chen2014}. The proof relies on the technical assumption that $t$-valley-filling profiles exist, which tends to be satisfied with high penetration of deferrable loads. However, in Section \ref{s.simulation}, it is shown that the concentration phenomenon still holds in real data traces when this assumption is removed.} 

%To build intuition, the tail probability bound of $V$ in \eqref{tail-bound} can be simplified for two different regimes of $t$ as
%\begin{equation}
%\label{intuition}
%\mathbb{P}(V - \ex V > t) \le
%\begin{cases}
%\exp\left(\frac{-t^2}{48\epsilon^2\lambda_1\ex V}\right), & t < \ex V \\
%\exp\left(\frac{-t}{48\epsilon^2\lambda_1}\right), & t \ge \ex V.
%\end{cases}
%\end{equation}
%Though looser than that in \eqref{tail-bound}, the tail bound in \eqref{intuition} highlights that $V$ has a Gaussian tail probability bound when $t < \ex V$ and an Exponential tail probability bound when $t \ge \ex V$.

Theorem \ref{thm: tail-bound} implies that the actual performance of Algorithm \ref{algorithm: uncertainty} does not deviate much from its mean. To illustrate this, consider the following example:
\begin{example}
Suppose that the baseload prediction is precise, i.e., $\epsilon_2=0$. Then the average load variation is
    \[ \ex[V] = \frac{s^2}{T} \sum_{t=2}^T\frac{1}{t} \approx s^2\ln T/T \]
and the tail bound in Theorem \ref{thm: tail-bound} can be simplified as
    \[ \mathbb{P}(V - \ex{V} > c\ex V) \le  \exp\left(-\frac{c^2}{2 + c}\frac{s^2}{16\epsilon^2 }\right). \]
Recall that constant $s$ is the variance of $a$ and constant $\epsilon$ is the maximum deviation of $a$ from its mean. The above expression shows that, with high probability, $V$ is at most a constant $c+1$ times of its mean $\ex V$.
\end{example}

More generally, the quantity $\lambda_1$ controls the decaying speed of the tail bound in \eqref{tail-bound}: the smaller $\lambda_1$, the faster the tail bound $\prob(V-\ex V>t)$ decays in $t$, and the load variation $V$ achieved by Algorithm \ref{algorithm: uncertainty} concentrates sharper around its mean $\ex V$. The following corollary highlights that $\lambda_1$ tends to 0 as $T$ increases, provided that the error correlation $f(t)$ decays fast enough in $t$. Note that the condition on $f$ is the same for Corollary \ref{cor: lambda1} and Proposition \ref{prop: large T}.

\begin{corollary}
\label{cor: lambda1}
Under the assumptions of Theorem \ref{thm: tail-bound}, if the error correlation $f \sim O(t^{-\frac{1}{2}-\alpha})$ for some $\alpha >0$, then $\lambda_1 \rightarrow 0$ as $T \rightarrow \infty$.
\end{corollary}

A detailed proof of Theorem \ref{thm: tail-bound} is included in the technical report \cite{chen2014}. Note that the bound we obtained in Theorem \ref{thm: tail-bound} is much sharper than the Markov and Chebyshev bounds for large $t$. This is done by controlling the moment generating function of $V$ using the Log-Sobolev inequality similar to the technique used in \cite{Bou09}.

\subsection{Bounds on the variance}

To further understand the scale of typical load variation $V$ under Algorithm \ref{algorithm: uncertainty}, it is useful to also study its variance.  In addition, the form of the variance highlights the impact of the tight concentration shown in Theorem \ref{thm: tail-bound}.

\begin{theorem}
Suppose a $t$-valley-filling solution exists for $t=1,2,\ldots,T$, and prediction errors are bounded by $\epsilon_1$ and $\epsilon_2$ as in \eqref{infinity norm}. Then the variance $\var(V)$ of $V$ obtained by Algorithm \ref{algorithm: uncertainty} is bounded above by
\begin{equation}
\label{upper bound on variance}
\var(V) \le \left(\frac{4\epsilon_1 s\ln T}{T}\right)^2 +  \left(\frac{4\epsilon_2 \sigma}{T^2} \sum_{t=0}^{T-1} F^2(t) \frac{T-t+1}{t+1}\right)^2.
\end{equation}
\label{thm: variance}
\end{theorem}
To interpret this result, let $\overline{\var(V)}$ denote the upper bound on $\var(V)$ provided in \eqref{upper bound on variance}. Theorem \ref{thm: variance} implies that $\ex V$ and $\sqrt{\overline{\var(V)}}$ scale similarly with $T$. %In particular, the first term $\frac{s^2}{T}\sum_{t=2}^T\frac{1}{t}$ in $\ex V$ scales with $T$ as $\Omega(\ln T/T)$ while the first term $\left(4\epsilon_1 s\ln T/T\right)^2$ in $\overline{\var(V)}$ scales with $T$ as $\Omega\left( (\ln T/T)^2 \right)$, and the second terms in $\ex V$ and $\overline{\var(V)}$ have the same relationship. Hence, the standard deviation $\sqrt{\var(V)}$, which is upper bounded by $\sqrt{\overline{\var(V)}}$, is at most on the same scale as $\ex V$ as $T$ expands. 

It immediately follows from the Chebyshev inequality that $V$ can only deviate significantly from $\ex(V)$ with a small probability.

\begin{corollary}
Under the assumptions in Theorem \ref{thm: variance}, for $t>0$,
    \begin{align}
    & \mathbb{P}(|V - \ex{V}| > t) \nonumber\\
    &~\le~ \frac{1}{t^2} \left[ \left(\frac{4\epsilon_1 s\ln T}{T}\right)^2 +  \left(\frac{4\epsilon_2 \sigma}{T^2} \sum_{\tau=0}^{T-1} F^2(\tau) \frac{T-\tau+1}{\tau+1}\right)^2 \right].
    \label{chebyshev}
    \end{align}
\end{corollary}

While the tail bound \eqref{tail-bound} in Theorem \ref{thm: tail-bound} scales at least exponentially in $t$, the Chebyshev inequality only provides a tail bound \eqref{chebyshev} that scales inverse quadratically in $t$. Hence for large $t$, \eqref{tail-bound} provides a much tighter tail bound. However for small values of $t$, the tail bound \eqref{chebyshev} is usually tighter since the variance $\var(V)$ is well estimated in \eqref{upper bound on variance}.

Furthermore, the variance $\var(V)$ vanishes as $T$ expands, provided that $f(t)$ decays sufficiently fast as $t$ grows, as formally stated in the following corollary.

\begin{corollary}
Under the assumptions of Theorem \ref{thm: variance}, if the error correlation $f \sim O(t^{-\frac{1}{2}-\alpha})$ for some $\alpha >0$, then $\var(V) \rightarrow 0$ as $T \rightarrow \infty$.
%\[\mathrm{Var}(V) = O\left(\frac{1}{T^{4\alpha}}\right).\]
\end{corollary}

\noindent Note that the condition on $f$ parallels that in Proposition \ref{prop: large T}.

\subsection{A case study}
\label{s.simulation}

Theorems \ref{thm: tail-bound} and \ref{thm: variance} provide theoretical guarantees that the load variance $V$ obtained by Algorithm \ref{algorithm: uncertainty} concentrates around its mean, if prediction errors are bounded as in \eqref{infinity norm} and error correlation decays sufficiently fast (c.f. Corollary 2). Thus, they give the intuition that the expected performance of Algorithm \ref{algorithm: uncertainty} is a useful metric to focus on, and does indeed give an indication of the ``typical'' performance of the algorithm.

However, our analysis is based on the assumption that a $t$-valley-filling solution exists, which relies on the penetration of deferrable load being high enough.  This is a necessary technical assumption for our analysis, and has been used by the previous analysis of Algorithm \ref{algorithm: uncertainty} as well, e.g., \cite{gan2013}.

Given this assumption in the analytic results, it is important to understand the robustness of the results to this assumption. To that end, here we provide a case study to demonstrate that this intuition is robust to the $t$-valley-filling assumption.

In our case study, we mimic the setting of \cite{gan2013}, where an average-case analysis of Algorithm \ref{algorithm: uncertainty} is performed.  In particular,  we use 24 hour residential load trace in the Southern California Edison (SCE) service area averaged over the year 2012 and 2013 \cite{sce_data} as the non-deferrable load, and wind power generation data from the Alberta Electric System Operator from 2004 to 2012 \cite{aeso_data}. The wind power generation data is scaled so that its average over 9 years corresponds to 30\% penetration level, and pick the wind generation of a random day as renewable during each run. We generate random prediction error in baseload and arrival of deferrable load similar to \cite{gan2013}.

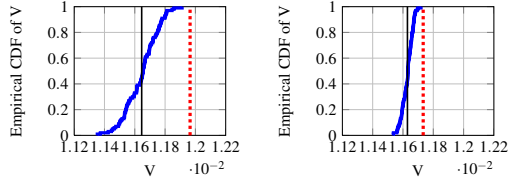
\begin{figure}[t]
	\begin{subfigure}[b]{0.2\textwidth}
	\resizebox{\columnwidth}{!}{
		% This file was created by matlab2tikz v0.4.7 running on MATLAB 8.1.
% Copyright (c) 2008--2014, Nico Schlömer <nico.schloemer@gmail.com>
% All rights reserved.
% Minimal pgfplots version: 1.3
% 
% The latest updates can be retrieved from
%   http://www.mathworks.com/matlabcentral/fileexchange/22022-matlab2tikz
% where you can also make suggestions and rate matlab2tikz.
% 
\begin{tikzpicture}

\begin{axis}[%
width=\textwidth,
height=0.85\textwidth,
unbounded coords=jump,
scale only axis,
xmin=0.0112,
xmax=0.0122,
xlabel={V},
xmajorgrids,
ymin=0,
ymax=1,
ylabel={Empirical CDF of V},
ymajorgrids
]
\addplot [color=blue,solid,line width=2.0pt,forget plot]
  table[row sep=crcr]{%
-inf	0\\
0.0113525833951026	0\\
0.0113525833951026	0.01\\
0.0113729896365413	0.01\\
0.0113729896365413	0.02\\
0.0114246228483665	0.02\\
0.0114246228483665	0.03\\
0.0114459010630518	0.03\\
0.0114459010630518	0.04\\
0.0114531888188776	0.04\\
0.0114531888188776	0.05\\
0.0114641957336751	0.05\\
0.0114641957336751	0.06\\
0.011474459025968	0.06\\
0.011474459025968	0.07\\
0.0114922980301803	0.07\\
0.0114922980301803	0.08\\
0.011493767465365	0.08\\
0.011493767465365	0.09\\
0.0114996443751665	0.09\\
0.0114996443751665	0.1\\
0.0115015898775566	0.1\\
0.0115015898775566	0.11\\
0.0115081430161832	0.11\\
0.0115081430161832	0.12\\
0.0115158645360207	0.12\\
0.0115158645360207	0.13\\
0.0115284571417281	0.13\\
0.0115284571417281	0.14\\
0.0115320596079824	0.14\\
0.0115320596079824	0.15\\
0.011532289025468	0.15\\
0.011532289025468	0.16\\
0.0115323822356439	0.16\\
0.0115323822356439	0.17\\
0.0115343971787149	0.17\\
0.0115343971787149	0.18\\
0.0115390405386823	0.18\\
0.0115390405386823	0.19\\
0.0115397576941633	0.19\\
0.0115397576941633	0.2\\
0.0115437279780675	0.2\\
0.0115437279780675	0.21\\
0.0115451479322362	0.21\\
0.0115451479322362	0.22\\
0.0115467452674502	0.22\\
0.0115467452674502	0.23\\
0.0115504697392484	0.23\\
0.0115504697392484	0.24\\
0.011554614247589	0.24\\
0.011554614247589	0.25\\
0.0115605648272637	0.25\\
0.0115605648272637	0.26\\
0.0115608413002056	0.26\\
0.0115608413002056	0.27\\
0.0115695992934989	0.27\\
0.0115695992934989	0.28\\
0.0115720399277873	0.28\\
0.0115720399277873	0.29\\
0.0115847986887678	0.29\\
0.0115847986887678	0.3\\
0.011596188939324	0.3\\
0.011596188939324	0.31\\
0.011596505176267	0.31\\
0.011596505176267	0.32\\
0.0115984242345061	0.32\\
0.0115984242345061	0.33\\
0.011611387390968	0.33\\
0.011611387390968	0.34\\
0.0116124881372171	0.34\\
0.0116124881372171	0.35\\
0.0116124983565246	0.35\\
0.0116124983565246	0.36\\
0.0116132399600568	0.36\\
0.0116132399600568	0.37\\
0.0116147631575706	0.37\\
0.0116147631575706	0.38\\
0.0116163166921786	0.38\\
0.0116163166921786	0.39\\
0.0116330107365537	0.39\\
0.0116330107365537	0.4\\
0.0116343885854029	0.4\\
0.0116343885854029	0.41\\
0.0116392088349109	0.41\\
0.0116392088349109	0.42\\
0.0116397044311653	0.42\\
0.0116397044311653	0.43\\
0.0116432162704033	0.43\\
0.0116432162704033	0.44\\
0.0116463234276233	0.44\\
0.0116463234276233	0.45\\
0.0116480508669042	0.45\\
0.0116480508669042	0.46\\
0.0116502883248043	0.46\\
0.0116502883248043	0.47\\
0.0116509818972213	0.47\\
0.0116509818972213	0.48\\
0.0116534079163879	0.48\\
0.0116534079163879	0.49\\
0.0116534737955921	0.49\\
0.0116534737955921	0.5\\
0.0116554408661571	0.5\\
0.0116554408661571	0.51\\
0.0116565643720895	0.51\\
0.0116565643720895	0.52\\
0.0116579283341555	0.52\\
0.0116579283341555	0.53\\
0.0116615424865543	0.53\\
0.0116615424865543	0.54\\
0.011661667422526	0.54\\
0.011661667422526	0.55\\
0.0116622074177908	0.55\\
0.0116622074177908	0.56\\
0.0116655377768931	0.56\\
0.0116655377768931	0.57\\
0.011666342117178	0.57\\
0.011666342117178	0.58\\
0.0116687634534498	0.58\\
0.0116687634534498	0.59\\
0.011669424992898	0.59\\
0.011669424992898	0.6\\
0.011670484033401	0.6\\
0.011670484033401	0.61\\
0.0116805289364397	0.61\\
0.0116805289364397	0.62\\
0.0116813990658342	0.62\\
0.0116813990658342	0.63\\
0.0116930054601419	0.63\\
0.0116930054601419	0.64\\
0.0116943092269696	0.64\\
0.0116943092269696	0.65\\
0.0116949143592146	0.65\\
0.0116949143592146	0.66\\
0.0116969284889237	0.66\\
0.0116969284889237	0.67\\
0.0116974892938153	0.67\\
0.0116974892938153	0.68\\
0.0116979352466769	0.68\\
0.0116979352466769	0.69\\
0.011700769742817	0.69\\
0.011700769742817	0.7\\
0.0117061627529055	0.7\\
0.0117061627529055	0.71\\
0.0117108183248287	0.71\\
0.0117108183248287	0.72\\
0.0117188481395943	0.72\\
0.0117188481395943	0.73\\
0.0117193032304049	0.73\\
0.0117193032304049	0.74\\
0.0117204201113028	0.74\\
0.0117204201113028	0.75\\
0.0117226110353513	0.75\\
0.0117226110353513	0.76\\
0.0117344104401869	0.76\\
0.0117344104401869	0.77\\
0.0117378654256708	0.77\\
0.0117378654256708	0.78\\
0.0117428400586526	0.78\\
0.0117428400586526	0.79\\
0.0117484276681755	0.79\\
0.0117484276681755	0.8\\
0.0117505449905863	0.8\\
0.0117505449905863	0.81\\
0.0117532730990487	0.81\\
0.0117532730990487	0.82\\
0.0117594183900411	0.82\\
0.0117594183900411	0.83\\
0.0117605575028245	0.83\\
0.0117605575028245	0.84\\
0.0117609392689798	0.84\\
0.0117609392689798	0.85\\
0.0117677819744805	0.85\\
0.0117677819744805	0.86\\
0.0117698443449756	0.86\\
0.0117698443449756	0.87\\
0.0117728899387885	0.87\\
0.0117728899387885	0.88\\
0.0117748019857803	0.88\\
0.0117748019857803	0.89\\
0.0117782883435424	0.89\\
0.0117782883435424	0.9\\
0.0117798557283509	0.9\\
0.0117798557283509	0.91\\
0.0117987123817076	0.91\\
0.0117987123817076	0.92\\
0.0117991142299193	0.92\\
0.0117991142299193	0.93\\
0.0118037489087862	0.93\\
0.0118037489087862	0.94\\
0.0118057543291988	0.94\\
0.0118057543291988	0.95\\
0.0118102338124128	0.95\\
0.0118102338124128	0.96\\
0.0118107502652776	0.96\\
0.0118107502652776	0.97\\
0.0118458088497506	0.97\\
0.0118458088497506	0.98\\
0.011858807206157	0.98\\
0.011858807206157	0.99\\
0.0119122940467647	0.99\\
0.0119122940467647	1\\
inf	1\\
};
\addplot [color=black,solid,line width=1.0pt,forget plot]
  table[row sep=crcr]{%
0.0116455835745651	0\\
0.0116455835745651	1\\
};
\addplot [color=red,dotted,line width=2.0pt,forget plot]
  table[row sep=crcr]{%
0.0119649135188415	0\\
0.0119649135188415	1\\
};
\end{axis}
\end{tikzpicture}%
		}
		\caption{30\% prediction error}
		\label{fig:cdf_30}
	\end{subfigure} ~
	\begin{subfigure}[b]{0.2\textwidth}
		\resizebox{\columnwidth}{!}{
		% This file was created by matlab2tikz v0.4.7 running on MATLAB 8.1.
% Copyright (c) 2008--2014, Nico Schlömer <nico.schloemer@gmail.com>
% All rights reserved.
% Minimal pgfplots version: 1.3
% 
% The latest updates can be retrieved from
%   http://www.mathworks.com/matlabcentral/fileexchange/22022-matlab2tikz
% where you can also make suggestions and rate matlab2tikz.
% 
\begin{tikzpicture}

\begin{axis}[%
width=\textwidth,
height=0.85\textwidth,
unbounded coords=jump,
scale only axis,
xmin=0.0112,
xmax=0.0122,
xlabel={V},
xmajorgrids,
ymin=0,
ymax=1,
ylabel={Empirical CDF of V},
ymajorgrids
]
\addplot [color=blue,solid,line width=2.0pt,forget plot]
  table[row sep=crcr]{%
-inf	0\\
0.0115363594790196	0\\
0.0115363594790196	0.01\\
0.0115367861899823	0.01\\
0.0115367861899823	0.02\\
0.0115592745181267	0.02\\
0.0115592745181267	0.03\\
0.0115608989234607	0.03\\
0.0115608989234607	0.04\\
0.0115636677222696	0.04\\
0.0115636677222696	0.05\\
0.0115658872896832	0.05\\
0.0115658872896832	0.06\\
0.0115675920657202	0.06\\
0.0115675920657202	0.07\\
0.0115737233056744	0.07\\
0.0115737233056744	0.08\\
0.0115758482031767	0.08\\
0.0115758482031767	0.09\\
0.011580038781425	0.09\\
0.011580038781425	0.1\\
0.0115812244461478	0.1\\
0.0115812244461478	0.11\\
0.0115826079466653	0.11\\
0.0115826079466653	0.12\\
0.0115834340923207	0.12\\
0.0115834340923207	0.13\\
0.0115884801674049	0.13\\
0.0115884801674049	0.14\\
0.0115896179426683	0.14\\
0.0115896179426683	0.15\\
0.0115908735775694	0.15\\
0.0115908735775694	0.16\\
0.0115909909090904	0.16\\
0.0115909909090904	0.17\\
0.0115925473002221	0.17\\
0.0115925473002221	0.18\\
0.0115947147737343	0.18\\
0.0115947147737343	0.19\\
0.0115953713085712	0.19\\
0.0115953713085712	0.2\\
0.01159595168655	0.2\\
0.01159595168655	0.21\\
0.011596169342608	0.21\\
0.011596169342608	0.22\\
0.0115989809826196	0.22\\
0.0115989809826196	0.23\\
0.0116010379220581	0.23\\
0.0116010379220581	0.24\\
0.011601167876655	0.24\\
0.011601167876655	0.25\\
0.0116013152369427	0.25\\
0.0116013152369427	0.26\\
0.0116044912344033	0.26\\
0.0116044912344033	0.27\\
0.011604739210337	0.27\\
0.011604739210337	0.28\\
0.011610584498867	0.28\\
0.011610584498867	0.29\\
0.0116110854380088	0.29\\
0.0116110854380088	0.3\\
0.011612571299852	0.3\\
0.011612571299852	0.31\\
0.0116145902439102	0.31\\
0.0116145902439102	0.32\\
0.0116148893158644	0.32\\
0.0116148893158644	0.33\\
0.0116153859614723	0.33\\
0.0116153859614723	0.34\\
0.0116173875611161	0.34\\
0.0116173875611161	0.35\\
0.0116178419477124	0.35\\
0.0116178419477124	0.36\\
0.0116211620838106	0.36\\
0.0116211620838106	0.37\\
0.0116212240320073	0.37\\
0.0116212240320073	0.38\\
0.0116246964435015	0.38\\
0.0116246964435015	0.39\\
0.0116248036361145	0.39\\
0.0116248036361145	0.4\\
0.0116275449814457	0.4\\
0.0116275449814457	0.41\\
0.0116290217858165	0.41\\
0.0116290217858165	0.42\\
0.0116292361477082	0.42\\
0.0116292361477082	0.43\\
0.0116301508329371	0.43\\
0.0116301508329371	0.44\\
0.0116308277654962	0.44\\
0.0116308277654962	0.45\\
0.0116309615128208	0.45\\
0.0116309615128208	0.46\\
0.0116312890467372	0.46\\
0.0116312890467372	0.47\\
0.0116319314615789	0.47\\
0.0116319314615789	0.48\\
0.0116319968195855	0.48\\
0.0116319968195855	0.49\\
0.0116323477566743	0.49\\
0.0116323477566743	0.5\\
0.0116327339903946	0.5\\
0.0116327339903946	0.51\\
0.0116328783196827	0.51\\
0.0116328783196827	0.52\\
0.0116336943578435	0.52\\
0.0116336943578435	0.53\\
0.0116348306522755	0.53\\
0.0116348306522755	0.54\\
0.0116353804321023	0.54\\
0.0116353804321023	0.55\\
0.0116354940550484	0.55\\
0.0116354940550484	0.56\\
0.0116356329898115	0.56\\
0.0116356329898115	0.57\\
0.0116366547443206	0.57\\
0.0116366547443206	0.58\\
0.0116375521314752	0.58\\
0.0116375521314752	0.59\\
0.0116391069855724	0.59\\
0.0116391069855724	0.6\\
0.0116393892871511	0.6\\
0.0116393892871511	0.61\\
0.0116409222072158	0.61\\
0.0116409222072158	0.62\\
0.0116435615482639	0.62\\
0.0116435615482639	0.63\\
0.0116446620321212	0.63\\
0.0116446620321212	0.64\\
0.0116454314854682	0.64\\
0.0116454314854682	0.65\\
0.0116466293988005	0.65\\
0.0116466293988005	0.66\\
0.0116481041835037	0.66\\
0.0116481041835037	0.67\\
0.0116482992814723	0.67\\
0.0116482992814723	0.68\\
0.0116508555645616	0.68\\
0.0116508555645616	0.69\\
0.0116509688077192	0.69\\
0.0116509688077192	0.7\\
0.011651305001819	0.7\\
0.011651305001819	0.71\\
0.0116513631246034	0.71\\
0.0116513631246034	0.72\\
0.0116526830355927	0.72\\
0.0116526830355927	0.73\\
0.0116557280501578	0.73\\
0.0116557280501578	0.74\\
0.0116567453518676	0.74\\
0.0116567453518676	0.75\\
0.0116568454260903	0.75\\
0.0116568454260903	0.76\\
0.0116575069796628	0.76\\
0.0116575069796628	0.77\\
0.0116586197477478	0.77\\
0.0116586197477478	0.78\\
0.0116586801104478	0.78\\
0.0116586801104478	0.79\\
0.0116597966760071	0.79\\
0.0116597966760071	0.8\\
0.0116649003831244	0.8\\
0.0116649003831244	0.81\\
0.0116652236524754	0.81\\
0.0116652236524754	0.82\\
0.0116662273197307	0.82\\
0.0116662273197307	0.83\\
0.0116668501416369	0.83\\
0.0116668501416369	0.84\\
0.0116673830395396	0.84\\
0.0116673830395396	0.85\\
0.0116693495728128	0.85\\
0.0116693495728128	0.86\\
0.0116695925930657	0.86\\
0.0116695925930657	0.87\\
0.0116697517916598	0.87\\
0.0116697517916598	0.88\\
0.0116727775197907	0.88\\
0.0116727775197907	0.89\\
0.0116728755909094	0.89\\
0.0116728755909094	0.9\\
0.0116737539125075	0.9\\
0.0116737539125075	0.91\\
0.0116741475594436	0.91\\
0.0116741475594436	0.92\\
0.0116757221789164	0.92\\
0.0116757221789164	0.93\\
0.0116771988916642	0.93\\
0.0116771988916642	0.94\\
0.0116794476753668	0.94\\
0.0116794476753668	0.95\\
0.0116806366847463	0.95\\
0.0116806366847463	0.96\\
0.0116837177275894	0.96\\
0.0116837177275894	0.97\\
0.0116943176622963	0.97\\
0.0116943176622963	0.98\\
0.011698721180159	0.98\\
0.011698721180159	0.99\\
0.011719058998002	0.99\\
0.011719058998002	1\\
inf	1\\
};
\addplot [color=black,solid,line width=1.0pt,forget plot]
  table[row sep=crcr]{%
0.0116293903304238	0\\
0.0116293903304238	1\\
};
\addplot [color=red,dotted,line width=2.0pt,forget plot]
  table[row sep=crcr]{%
0.011733388538984	0\\
0.011733388538984	1\\
};
\end{axis}
\end{tikzpicture}%
		}
	\caption{10\% prediction error}
		\label{fig:cdf_10}
	\end{subfigure}
\caption{\textit{The empirical cumulative distribution function of the load variance under Algorithm \ref{algorithm: uncertainty} over 24 hour control horizon using real data. The red line represents the analytic bound on the 90\% confidence interval computed from Theorem \ref{thm: tail-bound}, and the black line shows the empirical mean.}}
\label{fig:result}
\end{figure}

Given this setting, we simulate 100 instances in each scenario and compare the results with the Theorems \ref{thm: tail-bound}. The results are shown in Fig. \ref{fig:result} where we plot the cumulative distribution (CDF) of the load variance produced by Algorithm \ref{algorithm: uncertainty} under two different scenarios.  Specifically, in Fig. \ref{fig:cdf_30}, we assume the prediction error in wind power generation is $30\%$, and in Fig. \ref{fig:cdf_10}, we assume the prediction error is $10\%$. We plot the CDF on the same scale in both plots and additionally show an  analytic bound on the $90\%$ confidence interval computed from Theorem \ref{thm: tail-bound}.  For both cases, the results highlight a strong concentration around the mean, and the analytic bound from Theorem \ref{thm: tail-bound} is valid despite the fact that the $t$-valley-filling assumption is not satisfied.   Further, note that the analytic bound is much tighter when prediction error is small, which coincides the statement of Theorem \ref{thm: tail-bound}.

\section{Conclusion}
We have studied a promising algorithm for direct control demand response: model predictive deferrable load control.  In particular, we have, for the first time, provided a distributional analysis of the algorithm and shown that the load variance is tightly concentrated around its mean.  Thus, our results highlight that the typical performance one should expect to see with model predictive deferrable load control is not-too-different from the average-case analysis.  Importantly, the proof technique we develop may be useful for the analysis of model predictive control in more general settings as well.

The main limitation in our analysis (which is also true for the prior stochastic analysis of model predictive deferrable load control) is the assumption that a $t$-valley-filling solution exists. Practically, one can expect this to be satisfied if the penetration of deferrable loads is high; however, relaxing the need for this technical assumption remains an important challenge.  Interestingly, the numerical results we report here highlight that one should also expect a tight concentration in the case where a $t$-valley-filling solution does not exist.

\bibliographystyle{IEEEtran}
{
\bibliography{reference}

\begin{thebibliography}{10}
\providecommand{\url}[1]{#1}
\csname url@samestyle\endcsname
\providecommand{\newblock}{\relax}
\providecommand{\bibinfo}[2]{#2}
\providecommand{\BIBentrySTDinterwordspacing}{\spaceskip=0pt\relax}
\providecommand{\BIBentryALTinterwordstretchfactor}{4}
\providecommand{\BIBentryALTinterwordspacing}{\spaceskip=\fontdimen2\font plus
\BIBentryALTinterwordstretchfactor\fontdimen3\font minus
  \fontdimen4\font\relax}
\providecommand{\BIBforeignlanguage}[2]{{%
\expandafter\ifx\csname l@#1\endcsname\relax
\typeout{** WARNING: IEEEtran.bst: No hyphenation pattern has been}%
\typeout{** loaded for the language `#1'. Using the pattern for}%
\typeout{** the default language instead.}%
\else
\language=\csname l@#1\endcsname
\fi
#2}}
\providecommand{\BIBdecl}{\relax}
\BIBdecl

\bibitem{dr_mechanisms}
M.~H. Albadi and E.~El-Saadany, ``Demand response in electricity markets: An
  overview,'' in \emph{Power Engineering Society General Meeting, 2007. IEEE},
  June 2007, pp. 1--5.

\bibitem{system_loss}
E.~Sortomme, M.~Hindi, S.~MacPherson, and S.~Venkata, ``Coordinated charging of
  plug-in hybrid electric vehicles to minimize distribution system losses,''
  \emph{Smart Grid, IEEE Transactions on}, vol.~2, no.~1, pp. 198--205, March
  2011.

\bibitem{gan2013}
L.~Gan, A.~Wierman, U.~Topcu, N.~Chen, and S.~H. Low, ``Real-time deferrable
  load control: handling the uncertainties of renewable generation,'' in
  \emph{Proceedings of the fourth international conference on Future energy
  systems}.\hskip 1em plus 0.5em minus 0.4em\relax ACM, 2013, pp. 113--124.

\bibitem{Acha10}
S.~Acha, T.~C. Green, and N.~Shah, ``Effects of optimised plug-in hybrid
  vehicle charging strategies on electric distribution network losses,'' in
  \emph{Transmission and Distribution Conference and Exposition, 2010 IEEE
  PES}.\hskip 1em plus 0.5em minus 0.4em\relax IEEE, 2010, pp. 1--6.

\bibitem{Mets10}
K.~Mets, T.~Verschueren, W.~Haerick, C.~Develder, and F.~De~Turck, ``Optimizing
  smart energy control strategies for plug-in hybrid electric vehicle
  charging,'' in \emph{Network Operations and Management Symposium Workshops
  (NOMS Wksps), 2010 IEEE/IFIP}.\hskip 1em plus 0.5em minus 0.4em\relax IEEE,
  2010, pp. 293--299.

\bibitem{ilic2002potential}
M.~Ilic, J.~W. Black, and J.~L. Watz, ``Potential benefits of implementing load
  control,'' in \emph{Power Engineering Society Winter Meeting, 2002. IEEE},
  vol.~1.\hskip 1em plus 0.5em minus 0.4em\relax IEEE, 2002, pp. 177--182.

\bibitem{Ma10}
Z.~Ma, D.~Callaway, and I.~Hiskens, ``Decentralized charging control for large
  populations of plug-in electric vehicles,'' in \emph{Decision and Control
  (CDC), 2010 49th IEEE Conference on}.\hskip 1em plus 0.5em minus 0.4em\relax
  IEEE, 2010, pp. 206--212.

\bibitem{Gan12}
L.~Gan, U.~Topcu, and S.~H. Low, ``Stochastic distributed protocol for electric
  vehicle charging with discrete charging rate,'' in \emph{Power and Energy
  Society General Meeting, 2012 IEEE}.\hskip 1em plus 0.5em minus 0.4em\relax
  IEEE, 2012, pp. 1--8.

\bibitem{qin2003survey}
S.~J. Qin and T.~A. Badgwell, ``A survey of industrial model predictive control
  technology,'' \emph{Control engineering practice}, vol.~11, no.~7, pp.
  733--764, 2003.

\bibitem{camacho2013model}
E.~F. Camacho and C.~B. Alba, \emph{Model predictive control}.\hskip 1em plus
  0.5em minus 0.4em\relax Springer, 2013.

\bibitem{kwon2006receding}
W.~H. Kwon and S.~H. Han, \emph{Receding horizon control: model predictive
  control for state models}.\hskip 1em plus 0.5em minus 0.4em\relax Springer,
  2006.

\bibitem{bemporad1999robust}
A.~Bemporad and M.~Morari, ``Robust model predictive control: A survey,'' in
  \emph{Robustness in identification and control}.\hskip 1em plus 0.5em minus
  0.4em\relax Springer, 1999, pp. 207--226.

\bibitem{conejo2010real}
A.~J. Conejo, J.~M. Morales, and L.~Baringo, ``Real-time demand response
  model,'' \emph{Smart Grid, IEEE Transactions on}, vol.~1, no.~3, pp.
  236--242, 2010.

\bibitem{roos1998industrial}
J.~Roos and I.~Lane, ``Industrial power demand response analysis for one-part
  real-time pricing,'' \emph{Power Systems, IEEE Transactions on}, vol.~13,
  no.~1, pp. 159--164, 1998.

\bibitem{chen2012iems}
S.~Chen and L.~Tong, ``iems for large scale charging of electric vehicles:
  Architecture and optimal online scheduling,'' in \emph{Smart Grid
  Communications (SmartGridComm), 2012 IEEE Third International Conference
  on}.\hskip 1em plus 0.5em minus 0.4em\relax IEEE, 2012, pp. 629--634.

\bibitem{li2011line}
Q.~Li, T.~Cui, R.~Negi, F.~Franchetti, and M.~D. Ilic, ``On-line decentralized
  charging of plug-in electric vehicles in power systems,'' \emph{arXiv
  preprint arXiv:1106.5063}, 2011.

\bibitem{ng1998direct}
K.-H. Ng and G.~B. Sheble, ``Direct load control-a profit-based load management
  using linear programming,'' \emph{Power Systems, IEEE Transactions on},
  vol.~13, no.~2, pp. 688--694, 1998.

\bibitem{gan2011optimal}
L.~Gan, U.~Topcu, and S.~Low, ``Optimal decentralized protocol for electric
  vehicle charging,'' in \emph{Decision and Control and European Control
  Conference (CDC-ECC), 2011 50th IEEE Conference on}.\hskip 1em plus 0.5em
  minus 0.4em\relax IEEE, 2011, pp. 5798--5804.

\bibitem{ev_model}
K.~Clement-Nyns, E.~Haesen, and J.~Driesen, ``The impact of charging plug-in
  hybrid electric vehicles on a residential distribution grid,'' \emph{Power
  Systems, IEEE Transactions on}, vol.~25, no.~1, pp. 371--380, Feb 2010.

\bibitem{smart_home}
M.~Pedrasa, T.~Spooner, and I.~MacGill, ``Coordinated scheduling of residential
  distributed energy resources to optimize smart home energy services,''
  \emph{Smart Grid, IEEE Transactions on}, vol.~1, no.~2, pp. 134--143, Sept
  2010.

\bibitem{lee1997worst}
J.~a. Lee and Z.~Yu, ``Worst-case formulations of model predictive control for
  systems with bounded parameters,'' \emph{Automatica}, vol.~33, no.~5, pp.
  763--781, 1997.

\bibitem{lin2012online}
M.~Lin, Z.~Liu, A.~Wierman, and L.~L. Andrew, ``Online algorithms for
  geographical load balancing,'' in \emph{Green Computing Conference (IGCC),
  2012 International}.\hskip 1em plus 0.5em minus 0.4em\relax IEEE, 2012, pp.
  1--10.

\bibitem{chen2014}
N.~{Chen}, L.~{Gan}, S.~H. {Low}, and A.~{Wierman}, ``{Distributional Analysis
  for Model Predictive Deferrable Load Control},'' \emph{ArXiv e-prints}, Mar.
  2014.

\bibitem{Bou09}
S.~Boucheron, G.~Lugosi, P.~Massart \emph{et~al.}, ``On concentration of
  self-bounding functions,'' \emph{Electronic Journal of Probability}, vol.~14,
  no.~64, pp. 1884--1899, 2009.

\bibitem{sce_data}
``Southern california edison dynamic load profiles,''
  \url{https://www.sce.com/wps/portal/home/regulatory/load-profiles}, 2013.

\bibitem{aeso_data}
``Alberta eelctric system operator. wind power and alberta internal load
  data,'' \url{http://www.aeso.ca/gridoperations/20544.html}, 2012.

\bibitem{ledoux1999}
M.~Ledoux, ``Concentration of measure and logarithmic sobolev inequalities,''
  in \emph{Seminaire de probabilites XXXIII}.\hskip 1em plus 0.5em minus
  0.4em\relax Springer, 1999, pp. 120--216.

\end{thebibliography}
}

\appendix
\section{Proof of Proposition \ref{prop: worst-case}}
It has been computed in \cite{gan2013} that the load variance $V$ obtained by Algorithm \ref{algorithm: uncertainty} is composed of two parts:
	\[ V = V_1 + V_2, \]
where
	\begin{align*}
	V_1
	&~:=~ \frac{1}{T}\sum_{t=1}^T\left[\sum_{\tau=1}^t \frac{\tau-1}{T(T-\tau+1)}(a(\tau)-\lambda) \right. \\
	&\qquad\qquad\left.- \sum_{\tau=t+1}^T \frac{1}{T} (a(\tau)-\lambda) \right]^2
	\end{align*}
is the variance due to the prediction error on deferrable load and
	\begin{align*}
	V_2
	&~:=~ \frac{1}{T}\sum_{t=1}^T\left[\sum_{\tau=1}^t \frac{\tau-1}{T(T-\tau+1)}	 e(\tau)F(T-\tau) \right. \\
	&\qquad\qquad \left.- \sum_{\tau=t+1}^T \frac{1}{T} e(\tau)F(T-\tau)\right]^2
	\end{align*}
is the variance due to the prediction error on baseload. Now we compute the worst-case $V_1$ and $V_2$ under the bounded prediction error assumption \eqref{infinity norm}.

We start with computing the worst-case $V_1$. Let $x(\tau):=a(\tau)-\lambda$ for $\tau=1,2,\ldots,T$, then
	\begin{align*}
	V_1 &~=~ \frac{1}{T}\sum_{t=1}^T\left[\sum_{\tau=1}^t \frac{\tau-1}{T(T-\tau+1)}x(\tau) - \sum_{\tau=t+1}^T \frac{1}{T} x(\tau) \right]^2 \\
	&~=~ \frac{1}{T}\sum_{t=1}^T\left[\sum_{\tau=1}^t \frac{1}{T-\tau+1} x(\tau) - \sum_{\tau=1}^T \frac{1}{T} x(\tau) \right]^2 \\
%	&~=~ \frac{1}{T}\sum_{t=1}^T\left\{
%	\left[\sum_{\tau=1}^t \frac{1}{T-\tau+1} x(\tau)\right]^2
%	+ \left[\sum_{\tau=1}^T \frac{1}{T} x(\tau)\right]^2 \right. \\
%	& \qquad\qquad \left. - 2 \sum_{\tau=1}^t \frac{1}{T-\tau+1} x(\tau) \sum_{s=1}^T \frac{1}{T} x(s)
%	\right\}\\
	%%%%%%%%%%
	&~=~ \frac{1}{T}\sum_{t=1}^T \left[\sum_{\tau=1}^t \frac{1}{T-\tau+1} x(\tau)\right]^2
	+ \frac{1}{T}\sum_{t=1}^T \left[\sum_{\tau=1}^T \frac{1}{T} x(\tau)\right]^2 \\
	& \qquad\qquad - \frac{2}{T}\sum_{t=1}^T \sum_{\tau=1}^t \frac{1}{T-\tau+1} x(\tau) \sum_{s=1}^T \frac{1}{T} x(s) \\
	&~=~ \frac{1}{T}\sum_{t=1}^T \left[\sum_{\tau=1}^t \frac{1}{T-\tau+1} x(\tau)\right]^2
	+ \left[\sum_{\tau=1}^T \frac{1}{T} x(\tau)\right]^2 \\
	& \qquad\qquad - \frac{2}{T^2}\sum_{s=1}^T x(s)\sum_{\tau=1}^T \sum_{t=\tau}^T \frac{1}{T-\tau+1} x(\tau) \\
	&~=~ \frac{1}{T}\sum_{t=1}^T \left[\sum_{\tau=1}^t \frac{1}{T-\tau+1} x(\tau)\right]^2
	+ \frac{1}{T^2} \left[\sum_{\tau=1}^T x(\tau)\right]^2 \\
	& \qquad\qquad - \frac{2}{T^2}\sum_{s=1}^T x(s)\sum_{\tau=1}^T x(\tau) \\
	&~=~ \frac{1}{T}\sum_{t=1}^T \left[\sum_{\tau=1}^t \frac{1}{T-\tau+1} x(\tau)\right]^2
	- \frac{1}{T^2} \left[\sum_{\tau=1}^T x(\tau)\right]^2.
	\end{align*}
The first term
	\begin{align*}
	& \frac{1}{T}\sum_{t=1}^T \left[\sum_{\tau=1}^t \frac{1}{T-\tau+1} x(\tau)\right]^2 \\
	&~=~ \frac{1}{T}\sum_{t=1}^T \sum_{\tau=1}^t \left[\frac{1}{T-\tau+1} x(\tau)\right]^2 \\
	&\qquad + \frac{2}{T}\sum_{t=1}^T \sum_{\tau=1}^t \frac{1}{T-\tau+1} x(\tau) \sum_{s=\tau+1}^t \frac{1}{T-s+1} x(s) \\
	%%%%%%%%%%
	&~=~ \frac{1}{T}\sum_{\tau=1}^T \sum_{t=\tau}^T \frac{1}{(T-\tau+1)^2} x^2(\tau) \\
	&\qquad + \frac{2}{T} \sum_{\tau=1}^T \sum_{s=\tau+1}^T  \sum_{t=s}^T  \frac{1}{T-\tau+1}   \frac{1}{T-s+1} x(\tau) x(s)\\
	%%%%%%%%%%
	&~=~ \frac{1}{T}\sum_{\tau=1}^T \frac{1}{T-\tau+1} x^2(\tau) \\
	&\qquad + \frac{2}{T} \sum_{\tau=1}^T \sum_{s=\tau+1}^T \frac{1}{T-\tau+1}x(\tau)  x(s) \\
	%%%%%%%%%%
	&~=~ \frac{1}{T}\sum_{\tau=1}^T \sum_{s=1}^T \frac{1}{T-\tau \wedge s +1} x(\tau)x(s)
	\end{align*}
where $a\wedge b:=\min\{a,b\}$ for $a,b\in\mathbb{R}$. Let the matrix $A\in\mathbb{R}^{T \times T}$ be given by
	\[ A_{\tau s} := \frac{T}{T-\tau \wedge s + 1} \]
for $\tau,s=1,2,\ldots,T$, i.e.,
	\[ A = \begin{bmatrix}
	\frac{T}{T} & \frac{T}{T} & \frac{T}{T} & \cdots & \frac{T}{T} \\
	\frac{T}{T} & \frac{T}{T-1} & \frac{T}{T-1} & \cdots & \frac{T}{T-1} \\
	\frac{T}{T} & \frac{T}{T-1} & \frac{T}{T-2} & \cdots & \frac{T}{T-2} \\
	\vdots & \vdots & \vdots & \ddots & \vdots \\
	\frac{T}{T} & \frac{T}{T-1} & \frac{T}{T-2} & \cdots & \frac{T}{1},
	\end{bmatrix}\]
then
	\begin{align*}
	V_1 = \frac{1}{T^2} x^T \left( A - \one\one^T\right) x
	\end{align*}
where the vector $x:=(x(1),x(2),\ldots,x(T))^T$. When prediction error is bounded as in \eqref{infinity norm}, one has $|x(t)| \leq \epsilon_1$ for all $t$, and therefore
	\begin{align*}
	V_1 &~=~ \frac{1}{T^2} \sum_{\tau=1}^T\sum_{s=1}^T\left(A_{\tau s}-1\right)x(\tau)x(s) \\
	&~\leq~ \frac{1}{T^2} \sum_{\tau=1}^T\sum_{s=1}^T \frac{\tau\wedge s-1}{T-\tau\wedge s+1} \epsilon_1^2
	\end{align*}
and the equality is attained if and only if $x(t)=\epsilon_1$ for all $t$, or $x(t)=-\epsilon_1$ for all $t$. Finally, we simplify the worst-case expression of $V_1$ as follows:
	\begin{align*}
	\sup_a~ V_1 &~=~ \frac{1}{T^2} \sum_{\tau=1}^T\sum_{s=1}^T \frac{\tau\wedge s-1}{T-\tau\wedge s+1} \epsilon_1^2 \\
	&~=~ \frac{\epsilon_1^2}{T^2} \sum_{k=1}^T \frac{k-1}{T-k+1}(2T+1-2k) \\
	&~=~ \epsilon_1^2 \left(1- \frac{1}{T}\sum_{k=1}^T\frac{1}{k}\right)
	~\approx~ \epsilon_1^2 \left(1- \frac{\ln T}{T}\right).
	\end{align*}

We proceed to compute the worst-case $V_2$. Using the same derivation, it can be computed that
	\[ V_2 = \frac{1}{T^2} y^T \left( A - \one\one^T\right) y \]
where
	\begin{align*}
	y &~:=~ (y(1),y(2),\ldots,y(T))^T, \\
	y(t) &~:=~ e(t)F(T-t), \quad t=1,2,\ldots,T.
	\end{align*}
It follows that
	\begin{align*}
	V_2 &~=~ \frac{1}{T^2} \sum_{\tau=1}^T\sum_{s=1}^T\left(A_{\tau s}-1\right)y(\tau)y(s) \\
	&~\leq~ \frac{1}{T^2} \sum_{\tau=1}^T\sum_{s=1}^T \frac{\tau\wedge s-1}{T-\tau\wedge s+1} \epsilon_2^2|F(T-\tau)F(T-s)|
	\end{align*}
and that the equality is attained if and only if $e(t)=\epsilon_2\cdot\mathrm{sgn}(F(T-t))$ for all $t$, or $e(t)=-\epsilon_2\cdot\mathrm{sgn}(F(T-t))$ for all $t$. Finally, we simplify the worst-case expression of $V_2$ as follows:
	\begin{align*}
	\sup_e~ V_2 &~=~ \frac{1}{T^2} \sum_{\tau=1}^T\sum_{s=1}^T \frac{\tau\wedge s-1}{T-\tau\wedge s+1} \epsilon_2^2|F(T-\tau)F(T-s)| \\
	&~=~ \frac{\epsilon_2^2}{T^2} \sum_{\tau=0}^{T-1}\sum_{s=0}^{T-1} \left(\frac{T}{\tau\vee s+1}-1\right) |F(\tau)F(s)|
	\end{align*}

To summarize, the worst-case load variance $V$ obtained by Algorithm \ref{algorithm: uncertainty} is
	\begin{align*}
	\sup_{a,e}~ V &~=~ \epsilon_1^2 \left(1- \frac{1}{T}\sum_{k=1}^T\frac{1}{k}\right) \\
	& \qquad + \frac{\epsilon_2^2}{T^2} \sum_{\tau=0}^{T-1}\sum_{s=0}^{T-1} \left(\frac{T}{\tau\vee s+1}-1\right) |F(\tau)F(s)|.
	\end{align*}

%The lower bound in the lemma can be obtained from the case where all prediction errors of the load arrival is equal to $d_1/2$, then
%\begin{align*}
%\sup_{a} V & \ge \frac{d_1^2}{4T}\sum_{t=1}^T\left(\sum_{\tau=1}^t\frac{\tau-1}{T(T-\tau+1)} - \sum_{\tau=t+1}^T\frac{1}{T}\right)^2 \\
%&=\frac{d_1^2}{4T^3}\sum_{t=1}^T\left(\sum_{\tau=1}^t\frac{T}{T-\tau+1} - T \right)^2 \\
%&=\frac{d_1^2}{4T}\sum_{t=1}^T\left(\sum_{\tau=1}^t\frac{1}{T-\tau+1}-1\right)^2 \\
%&=\frac{d_1^2}{4T}\left(\sum_{t=1}^T(\sum_{\tau=T-t+1}^T\frac{1}{\tau})^2 - T\right) \\
%&\ge\frac{d_1^2}{4T} \left(\sum_{t=1}^T(\int^T_{T-t+1} \frac{1}{u} du )^2 - T\right) \\
%&=\frac{d_1^2}{4T} \left(\sum_{k=1}^T(\ln(\frac{T}{k}))^2-T \right) \qedhere
%\end{align*}

\section{Proof of Theorem \ref{thm: tail-bound}}
\label{app: tail-bound}

The theorem relies on a variant of the Log-Sobolev inequality provided in the following lemma.
\begin{lemma}[Theorem 3.2, \cite{ledoux1999}]
Let $f:\mathbb{R}^n\mapsto\mathbb{R}$ be convex and $X$ be supported on $[-d/2, d/2]^n$, then
\begin{align}
&\quad \ex[\exp(f(X)) f(X) ] - \ex[\exp(f(X))] \log \ex[\exp(f(X))] \notag \\
&\le \frac{d^2}{2} \ex[\exp(f(X)) ||\nabla f(X)||^2].
\label{eqn: log-sobolev}
\end{align}
\label{lemma: log-sobolev}
\end{lemma}

If $f$ is further ``self-bounded'', then its tail probability can be bounded as in the following lemma.
\begin{lemma}
Let $f:\mathbb{R}^n\mapsto\mathbb{R}$ be convex and $X$ be supported on $[-d/2,d/2]^n$. If $\ex [f(X)]=0$ and $f$ satisfies the following self-bounding property
	\begin{equation}
	\label{self_bound} ||\nabla f||^2 \le af + b,
	\end{equation}
then the tail probability of $f(X)$ can be  bound as
\begin{equation}
\mathbb{P}\left\{f(X) >t\right\} \le \exp\left(\frac{-t^2}{2b + at}\right).
\end{equation}
\label{lemma: self-bound}
\end{lemma}

\begin{proof}%[Proof of Lemma \ref{lemma: self-bound}]
Denote the moment generating function of $f(X)$ by 
	\[ m(\theta) := \ex e^{\theta f(X)}, \qquad \theta >0. \]
The function $\theta f:\mathbb{R}^n\mapsto \mathbb{R}$ is convex, and therefore it follows from Lemma \ref{lemma: log-sobolev} that
	\begin{align*}
	& \ex\left[e^{\theta f}\theta f\right] - \ex\left[e^{\theta f}\right] \ln\ex\left[e^{\theta f}\right] \leq \frac{d^2}{2} \ex\left[e^{\theta f} ||\theta\nabla f||^2\right], \\
	& \theta m'(\theta) - m(\theta) \ln m(\theta) \leq \frac{1}{2}\theta^2d^2 \ex[e^{\theta f} ||\nabla f||^2].
	\end{align*}
According to the self-bounding property \eqref{self_bound}, one has
	\begin{align*}
	\theta m'(\theta) - m(\theta) \ln m(\theta)
	&~\leq~ \frac{1}{2}\theta^2d^2 \ex[e^{\theta f} (af+b)] \\
	&~=~ \frac{1}{2}\theta^2d^2 \left[ am'(\theta) + bm(\theta) \right].
	\end{align*}
Divide both sides by $\theta^2 m(\theta)$ to get
\[
\frac{d}{d\theta}\left[\left(\frac{1}{\theta} - \frac{ad^2}{2}\right)\ln m(\theta)\right] \le \frac{bd^2}{2}.
\]
Integrate both sides from 0 to $s$ to get
\[
\left.\left(\frac{1}{\theta} - \frac{ad^2}{2}\right)\ln m(\theta)~\right|_{\theta=0}^s \le \frac{1}{2} bd^2s
\]
for $s\geq0$. Noting that $m(0)=1$ and $m'(0)=\ex f=0$, one has
	\begin{align*}
	\lim_{\theta\rightarrow0^+} \left(\frac{1}{\theta} - \frac{ad^2}{2}\right)\ln m(\theta)
	%&~=~ \lim_{\theta\rightarrow0^+} \frac{1}{\theta} \ln m(\theta)                   -\lim_{\theta\rightarrow0^+} \frac{ad^2}{2}\ln m(\theta) \\
	%&~=~ \lim_{\theta\rightarrow0^+} \frac{m'(\theta)}{m(\theta)}                    -\frac{ad^2}{2}\ln m(0) \\
	~=~ 0,
	\end{align*}
and therefore
\begin{equation}
\label{moment generating function}
\left(\frac{1}{s} - \frac{ad^2}{2}\right)\ln m(s) \le \frac{1}{2} bd^2s
\end{equation}
for $s\geq0$. We can bound the tail probability $\prob\{f>t\}$ with the control \eqref{moment generating function} over the moment generating function $m(s)$.

In particular, one has
	\begin{align*}
	\prob\{f > t\}
	&~=~ \prob\left\{ e^{sf} > e^{st} \right\} ~\leq~ e^{-st} \ex \left[ e^{sf} \right] \\%\inf_{\theta>0} \exp\left\{-\theta t + \frac{bd^2\theta^2}{2-ad^2\theta}\right\},\\
	&~=~ \exp[-st + \ln m(s)] \\
	&~\leq~ \exp\left[-st + \frac{bd^2s^2}{2-asd^2}\right]
	\end{align*}
for $s\geq0$. Choose $s = t/(bd^2 + ad^2 t/2)$ to get
\[
\mathbb{P}\{f >t\} \le \exp\left(\frac{-t^2}{d^2(2b + at)}\right) .%\qedhere
\]
\end{proof}

%Now we proceed to prove Theorem \ref{thm: tail-bound}
\begin{proof}[Proof of Theorem \ref{thm: tail-bound}]

It has been computed in \cite{gan2013} that the load variance $V$ obtained by Algorithm \ref{algorithm: uncertainty} is composed of two parts:
	\[ V = V_1 + V_2 \]
where
	\begin{align*}
	V_1
	&~:=~ \frac{1}{T}\sum_{t=1}^T\left[\sum_{\tau=1}^t \frac{\tau-1}{T(T-\tau+1)}(a(\tau)-\lambda) \right. \\
	&\qquad\qquad\left.- \sum_{\tau=t+1}^T \frac{1}{T} (a(\tau)-\lambda) \right]^2
	\end{align*}
is the variance due to the prediction error on deferrable load and
	\begin{align*}
	V_2
	&~:=~ \frac{1}{T}\sum_{t=1}^T\left[\sum_{\tau=1}^t \frac{\tau-1}{T(T-\tau+1)}	 e(\tau)F(T-\tau) \right. \\
	&\qquad\qquad \left.- \sum_{\tau=t+1}^T \frac{1}{T} e(\tau)F(T-\tau)\right]^2
	\end{align*}
is the variance due to the prediction error on baseload.

Let $x(\tau) := a(\tau) - \lambda$ for $\tau=1,2,\ldots,T$, then
	\begin{align*}
	V_1 &~=~ \frac{1}{T}\sum_{t=1}^T\left[\sum_{\tau=1}^t \frac{\tau-1}{T(T-\tau+1)}x(\tau) - \sum_{\tau=t+1}^T \frac{1}{T} x(\tau)\right]^2 \\
	&~=~ \frac{1}{T}||Bx||_2^2
	\end{align*}
where the $T\times T$ matrix $B$ is given by
\[B_{t \tau} := \begin{cases}
\frac{\tau-1}{T(T-\tau + 1)} & \tau \le t \\
-\frac{1}{T} & \tau >t
\end{cases}, \qquad 1\leq t,\tau\leq T.\]
Similarly, the variance $V_2$ due to the prediction error on baseload can be written as
\[V_2 = g(e) = \frac{1}{T}||Ce||_2^2 \]
where the $T\times T$ matrix $C$ is given by
\[ C_{t\tau} := \begin{cases}
\frac{\tau-1}{T(T-\tau+1)}F(T-\tau), & \tau \le t \\
-\frac{1}{T}F(T-\tau), & \tau>t
\end{cases}\]
for $1\leq t,\tau\leq T$. Therefore, the load variance
	\[ V = V_1 + V_2 = \frac{1}{T} \|Ay\|_2^2 \]
where
	\[
	A = \begin{bmatrix}
	B & 0 \\
	0 & C
	\end{bmatrix}, \quad 
	y = \begin{bmatrix}
	x \\ e
	\end{bmatrix}. \]
	
Define a centered random variable
	\[ Z := h(y) := V - \ex V  = \frac{1}{T}||Ay||^2 - \ex V \]
and note that the function $h$ is convex. Let $\lambda_{\max}$ be the maximum eigenvalue of $A A^T/T$, then
\begin{align*}
||\nabla h(y)||^2
&~=~ \frac{4}{T^2}||A^T Ay||^2 ~=~ \frac{4}{T} (Ay)^T\left(\frac{A A^T}{T}\right)(Ay) \\
&~\leq~ \frac{4\lambda_{\max}}{T} (Ay)^T(Ay) ~=~ 4 \lambda_{\max}[h(y) + \ex V].
\end{align*}
According to the bounded prediction error assumption \eqref{infinity norm}, one has $|y|\leq\epsilon$ componentwise. Then, apply Lemma \ref{lemma: self-bound} to the random variable $Z$ to obtain
	\[\mathbb{P}\{Z >t\} \le \exp\left(-\frac{t^2}{16\lambda_{\max} \epsilon^2(2\ex V + t)}\right) \]
for $t>0$, i.e., 
	\[\mathbb{P}\{V-\ex V >t\} \le \exp\left(-\frac{t^2}{16\lambda_{\max} \epsilon^2(2\ex V + t)}\right) \]
for $t>0$. Finally, let $A' = A A^T/T$, $B' = B B^T /T$, and $C' = C C^T/T$, the largest eigenvalue $\lambda_{\max}$ of $A A^T/T $ can be bounded above as
\begin{align*}
\lambda_{\max} &= \max_{y} \frac{y^T A' y}{y^T y} = \max_{x,e}\frac{x^TB'x + e^T C' e}{x^T x + e^T e} \\
&\le \max_{x,e} \frac{ \lambda^{B'}_{\max} x^T x + \lambda^{C'}_{\max} e^T e}{x^Tx + e^Te} \\
&\le \max_{x,e} \frac{ \max(\lambda^{B'}_{\max},  \lambda^{C'}_{\max}) (x^Tx + e^T e)}{x^Tx + e^Te} \\ 
&=  \max(\lambda^{B'}_{\max},  \lambda^{C'}_{\max}) \\
&\le \max\left(\mbox{tr}\left(\frac{B B^T}{T}\right) , \mbox{tr}\left(\frac{C C^T}{T}\right)\right) \\
&= \max\left( \frac{\ln T}{T} , \frac{1}{T^2}\sum_{t=0}^{T-1}F^2(t)\frac{T-t-1}{t+1}  \right) ~= :~ \lambda_1.
%&\le \mbox{tr}\left(\frac{A A^T}{T} \right) = \mbox{tr}\left(\frac{B B^T}{T}\right) + \mbox{tr}\left(\frac{C C^T}{T}\right) \\
%& = \frac{1}{T}\left(\sum_{t=2}^T \frac{1}{t}\right) + \frac{1}{T^2}\sum_{t=0}^{T-1}F^2(t)\frac{T-t-1}{t+1} \\
%&\le \frac{\ln T}{T} + \frac{1}{T^2}\sum_{t=0}^{T-1}F^2(t)\frac{T-t-1}{t+1} ~= :~ \lambda_1,
\end{align*}

The last equality is because
\begin{align*}
& \qquad \mbox{tr}\left(BB^T\right) = \frac{1}{T} \sum_{i=1}^T (B B^T)_{ii} 
 = \sum_{i=1}^T \sum_{k=1}^T (B_{ki})^2 \\
&=\frac{1}{T^2} \sum_{i=1}^T \left(\sum_{k=1}^i \frac{(k-1)^2}{(T-k+1)^2}+ (T-i) \right) \\
&=\frac{1}{T^2} \sum_{k=1}^T ( \frac{(k-1)^2}{(T-k+1)} + \sum_{i=1}^T (T-i)\\
&=\frac{1}{T^2}\sum_{k=1}^T \frac{(T-k)^2}{k} + \sum_{k=1}^T \frac{(T-k)k}{k} \\
&=\sum_{k=2}^T\frac{1}{k} \le \ln T,
\end{align*}
and 
\begin{align*}
& \mathrm{tr}(C C^T) 
= \sum_{i=1}^T \left(\sum_{k=1}^T C_{ki}^2\right) \\
 = &\frac{1}{T^2}\sum_{i=1}^T \left(\sum_{k=1}^i \frac{(k-1)^2}{(T-k+1)^2}F^2(T-k) + \sum^T_{k=i+1}F^2(T-k)\right) \\
 = &\frac{1}{T^2}\left(\sum_{k=2}^T\frac{(k-1)^2}{T-k+1}F^2(T-k) + \sum^T_{k=2}(k-1)F^2(T-k)\right) \\
 = &\frac{1}{T} \sum^T_{k=2} F^2(T-k) \frac{k-1}{T-k+1}.   \qquad \qedhere
\end{align*}. \qedhere 
\end{proof}

\section{Proof of Theorem \ref{thm: variance} }

The derivation of the theorem is based on the following two lemma, which separates the cases when there is only one type of prediction error.

\begin{lemma}
If there is no prediction error in the base load, then the variance of the performance of Algorithm \ref{algorithm: uncertainty} is bounded by
\begin{equation}
\mathrm{var}(V) \le \left(4 \epsilon_1 s \frac{\ln T}{T}\right)^2.
\end{equation}
\label{lemma: var-deferrable-only}
\end{lemma}

\begin{lemma}
If there is no prediction error in the deferrable load, then the variance of the performance of Algorithm \ref{algorithm: uncertainty} is bounded by
\begin{equation}
\mathrm{var}(V) \le \left(4 \epsilon_2 \sigma \frac{1}{T^2} \sum_{t=0}^{T-1} F^2(t) \frac{T-t+1}{t+1}\right)^2.
\end{equation}
\label{lemma: var-base-only}
\end{lemma}

Firstly we will prove Lemma \ref{lemma: var-deferrable-only}, where we only consider prediction error in deferrable load.

\begin{proof}[Proof of Lemma \ref{lemma: var-deferrable-only}]

Let $x(\tau) = a(\tau) - \lambda$, then $x(\tau)$ is centered, with variance $s^2$. Let $x = (x(1), \ldots, x(T))$. From the results in \cite{gan2013} Lemma 1, we have
\[V_1 = \frac{1}{T}\sum_{t=1}^T\left(\sum_{\tau=1}^t \frac{\tau-1}{T(T-\tau+1)}x(\tau) - \sum_{\tau=t+1}^T \frac{1}{T} x(\tau)\right)^2\]
Define an auxilary matrix $B$ such that
\[B_{t \tau} : = \begin{cases}
\frac{\tau-1}{T(T-\tau + 1)} & \tau \le t \\
-\frac{1}{T} & \tau >t.
\end{cases}\]
Then we have
\[V_1 = f(x(1), x(2), \ldots, x(T)) = \frac{1}{T}||Bx||_2^2. \]
Hence $V_1 = f(x)$ is a convex function, by convex Poincar\'{e} inequality, we have
\begin{equation}
\mbox{var}(V_1) \le 4 \epsilon_1^2\ex[||\nabla f(x)||^2].
\label{eqn: poincare}
\end{equation}
Whereas
\begin{align*}
\ex \left[||\nabla f(x) ||^2 \right] &= \frac{4}{T^2} \ex \left[ ||B^T B x||^2 \right] \\
&\le \frac{4}{T^2} \lambda_{\max} (B^T B) \ex\left[||Bx||^2 \right] \\
&\le 4\mbox{tr}\left(\frac{1}{T} B^T B \right) \ex \left[\frac{1}{T}||Bx||^2\right] \\
&= 4 s^2 \left [ \mbox{tr}\left(\frac{1}{T} B^T B \right) \right]^2 \\
&\le 4 s^2 \left(\frac{\ln T}{T}\right)^2
\end{align*}

\end{proof}

Next we proof lemma \ref{lemma: var-base-only} the case where we only consider the prediction error in the base load.

 \begin{proof}[Proof of Lemma \ref{lemma: var-base-only}]
Let $e = (e(1), \ldots, e(T))$, when there is no prediction error in the deferrable load arrival, we have

\begin{align*}
V = &\frac{1}{T}\sum_{t=1}^T(\sum_{\tau=1}^t \frac{\tau-1}{T(T-\tau+1)}F(T-\tau) e(\tau) \\
& - \sum_{\tau=t+1}^T \frac{1}{T}F(T-\tau) e(\tau))^2.
\end{align*}

If we define an auxilary matrix $C$ such that
\[ C_{t\tau} = \begin{cases}
\frac{\tau-1}{T(T-\tau+1)}F(T-\tau), & \tau \le t \\
-\frac{1}{T}F(T-\tau), & \tau>t
\end{cases}\]
Then we have
\[V = g(e(1), e(2), \ldots, e(T)) = \frac{1}{T}||Ce||_2^2. \]
Hence $V = g(e)$ is a convex function in $e$. By similar argument as Lemma \ref{lemma: var-deferrable-only}
\begin{equation}
\mbox{var}(V) \le 4\epsilon_2^2\ex[||\nabla g(e)||^2].
\label{eqn: poincare}
\end{equation}
Whereas
\begin{align*}
\ex \left[||\nabla g(e) ||^2 \right] &= \frac{4}{T^2} \ex \left[ ||C^T C e||^2 \right] \\
&\le \frac{4}{T^2} \lambda_{\max} (C^T C) \ex\left[||Ce||^2 \right] \\
&\le 4\mbox{tr}\left(\frac{1}{T} C^T C \right) \ex \left[\frac{1}{T}||Ce||^2\right] \\
&= 4 \sigma^2 \left [ \mbox{tr}\left(\frac{1}{T} C^T C \right) \right]^2 \\
&=4 \sigma^2 \left(\frac{1}{T^2} \sum_{t=0}^{T-1} F^2(t) \frac{T-t+1}{t+1}\right)^2.
\end{align*}
\end{proof}

Next, we bring the two results together to get a proof of Theorem \ref{thm: variance}.

\begin{proof}[Proof of Theorem 2]
Let $V_1$ be the load variance without prediction error in base load and $V_2$ be the load variance without prediction error in the deferrable load.
\[V = V_1 + V_2 .\]

By independence of $x$ and $e$, the variance of $V$ is bounded by
\begin{align*}
& \mathrm{var}(V) = \mbox{var}(V_1) +\mbox{var}(V_2) \\
&\le  \left(\frac{4 \epsilon_1 s\ln T  }{T}\right)^2 + \left(\frac{ 4 \epsilon_2 \sigma}{T^2} \sum_{t=0}^{T-1} F^2(t) \frac{T-t+1}{t+1}\right)^2. 
\end{align*}

\end{proof}

\end{document}